\documentclass[12pt,leqno,a4paper]{amsart}
\usepackage{amssymb,enumerate}                  

\usepackage[usenames]{color}
\usepackage{hyperref}   

\DeclareMathOperator{\End}{End}
\DeclareMathOperator{\Inf}{Inf}
\DeclareMathOperator{\Res}{Res}
\DeclareMathOperator{\Syl}{Syl}    
\DeclareMathOperator{\Irr}{Irr}   
\DeclareMathOperator{\proj}{(proj)}

\newcommand{\PGL}{{\operatorname{PGL}}}
\newcommand{\PSL}{{\operatorname{PSL}}}
\newcommand{\PgammaL}{{\operatorname{P\Gamma L}}}
\newcommand{\SL}{{\operatorname{SL}}}

\newtheorem{thm}{Theorem}[section]
\newtheorem{cor}[thm]{Corollary}
\newtheorem{lem}[thm]{Lemma}
\newtheorem{prop}[thm]{Proposition}

\theoremstyle{definition}

\newtheorem{rem}[thm]{Remark}
\newtheorem{ex}[thm]{Example}
\newtheorem{ass}[thm]{Hypotheses}

\theoremstyle{remark}

\newcommand{\IF}{\mathbb{F}}

\newcommand{\IZ}{\mathbb{Z}}

\newcommand{\fA}{{\mathfrak{A}}}
\newcommand{\fS}{{\mathfrak{S}}}

\let\lra=\longrightarrow

\newcommand{\stmod}{\mathsf{stmod}}  
\newcommand{\indhg}[2]{\!\uparrow_{#1}^{#2}}   
\newcommand{\resgh}[2]{\!\downarrow^{#1}_{#2}}   

\begin{document}


\title{Endo-trivial modules for finite groups 
\\
with dihedral Sylow 2-subgroup}

\author{Shigeo Koshitani and Caroline Lassueur}
\address{Department of Mathematics, Graduate School of Science,
Chiba University, 1-33 Yayoi-cho, Inage-ku, Chiba, 263-8522, Japan.}
\email{koshitan@math.s.chiba-u.ac.jp}
\address{FB Mathematik, TU Kaiserslautern, Postfach 3049,
         67653 Kaiserslautern, Germany.}
\email{lassueur@mathematik.uni-kl.de}
\thanks{The first author was partially supported by the Japan Society for 
Promotion of Science (JSPS), Grant-in-Aid for Scientific Research
(C)23540007, 2011-2014 and (C)15K04776, 2015--2018.
The second author acknowledges partial financial 
support by SNSF Fellowship for Prospective 
Researchers PBELP2$_{-}$143516.}


\begin{abstract}
Let $k$ be an algebraically closed field of characteristic $p>0$
and $G$ a finite group.
We provide a description of the torsion subgroup $TT(G)$ 
of the finitely generated abelian group  
$T(G)$ of endo-trivial $kG$-modules when  $p=2$ and $G$ has 
a dihedral Sylow $2$-subgroup $P$.
We prove that, in the case $|P|\geq 8$,
$TT(G)\cong X(G)$ the group of one-dimensional $kG$-modules, 
except possibly when $G/O_{2'}(G)\cong \fA_6$, the alternating group of 
degree $6$; 
in which case $G$ may have $9$-dimensional simple torsion 
endo-trivial modules. 
We also prove a similar result in the case $|P|=4$, although the situation 
is more involved.
Our results complement the tame-representation type investigation 
of endo-trivial modules  
started by Carlson-Mazza-Th{\'e}venaz  in the cases of 
semi-dihedral 
and generalized quaternion Sylow 2-subgroups.
Furthermore we provide a general reduction result, valid at any prime $p$, 
to recover the structure of $TT(G)$ 
from the structure of $TT(G/H)$, where $H$ is a normal $p'$-subgroup  of~$G$.
\end{abstract}

\maketitle

\pagestyle{myheadings}
\markboth{\textsc{S. Koshitani and C. Lassueur}}
{\textsc{Endo-trivial modules for finite groups with dihedral Sylow 2-subgroup}}


\section{Introduction}\label{sec:intro} Let $G$ be a finite group and $k$ 
a field of prime characteristic $p$ dividing the order of $G$. 
A finitely generated $kG$-module $V$  is called {\it endo-trivial} if, 
as $kG$-modules, 
$${\mathrm{End}}_k(V) \cong k_G\oplus Q\,,$$ 
where $k_G$ is the trivial $kG$-module 
and $Q$ is a projective $kG$-module. The
tensor product over $k$ induces a group structure on the set of
isomorphism classes of indecomposable endo-trivial $kG$-modules, called
the group of endo-trivial modules and denoted by $T(G)$. This group is
finitely generated and it is of particular
interest in modular representation theory as it forms an important part
of the Picard group of self-equivalences of the stable category of
finitely generated $kG$-modules. In particular the self-equivalences of
Morita type are induced by tensoring 
with endo-trivial modules. 

As a matter of fact, endo-trivial modules 
have seen a considerable interest since defined by
Dade in 1978  \cite{Dade1978} as a by-product of 
the Dade-Glauberman-Nagao correspondence (see \cite[\S 5.12]{NT89}).
In \cite{Dade1978} a classification of the endo-trivial modules over 
finite abelian $p$-groups is established. 
Since then a full classification has been obtained
over finite $p$-groups through the joint efforts of several authors, 
see e.g. the survey article \cite{Th07} and the references
therein. Moreover many contributions towards a general classification 
of endo-trivial
modules  have  been obtained over the past ten years for several families of
finite groups 
(see e.g. 
\cite{CMN06, MT07, CMN09, CMT11, Ca12, NR12, CMT13, KL14, CaTh15, LaMaz15II} 
and the references therein).  
However, the problem of describing the structure of  $T(G)$ and its 
elements  for an arbitrary finite group $G$
remains open in general. In particular the problem of determining 
the structure of the torsion
subgroup of $T(G)$, denoted by $TT(G)$, is a resisting part of the problem.
\par
Provided that a Sylow $p$-subgroup $P$ of $G$ is neither cyclic, 
nor dihedral, nor semi-dihedral, the group $TT(G)$ coincides with the group 
$$K(G)=\ker \big(\Res^G_P:T(G)\lra T(P):[V]\mapsto[V\resgh{G}{P}] \big)\,.$$
(See Lemma~\ref{lem:TT(G)}.) In particular, the group $K(G)$ consists 
of the classes of the trivial source endo-trivial modules.

Our first main result holds at any prime characteristic $p$ 
and relates the structure of the group $K(G)$ to the structure of 
$K(G/H)$, where $H$ is a normal $p'$-subgroup of $G$.
\begin{thm}\label{thm:PreMain}\label{cor:KXK}
Let $G$ be a finite group. Assume that the $p$-rank of $G$ is 
at least $2$ and that $G$ has no strongly $p$-embedded subgroups.
Let $H \vartriangleleft G$ with $p\,{\not|}\,|H|$, and set 
$\overline G:=G/H$.
If  ${\mathrm{H}}^2(\overline G,k^\times)=1$, then 
$$K(G)=X(G)+\Inf_{\overline{G}}^{G}(K(\overline{G})) 
\cong X(G)+K(\overline{G})\,.$$
\end{thm}

However, the main objective of this article is the determination of 
the structure of the group $T(G)$, when 
the Sylow $2$-subgroups of $G$ are dihedral $2$-groups.
An investigation of 
endo-trivial modules in finite- and tame-representation types 
was started by Mazza-Th\'{e}venaz \cite{MT07} 
in the cyclic case, Carlson-Mazza-Th\'{e}venaz \cite{CMT13} in the 
generalized quaternion and 
semi-dihedral cases, and continued by the authors \cite{KL14} 
in the Klein-four case. Therefore the 
dihedral case was the last untreated  tame-representation type case.
However, we emphasize that \cite{CMT13} does not provide a description 
of the structure of the group $K(G)$ in the semi-dihedral case, 
and which is still an open question.
\par
Since the group $T(G)$ is finitely generated 
(see \cite[Corollary 2.5]{CMN06}), 
we may write $T(G)=TT(G)\oplus TF(G)$, where $TF(G)$ denotes 
a free abelian complement of $TT(G)$ in $T(G)$.
We recall that, when a Sylow $2$-subgroup of $G$ is 
dihedral of order of at least~$8$, 
the $\IZ$-rank, as well as generators for the torsion-free part 
of $T(G)$ are known since 1980's.  
More precisely, since the $2$-rank of $G$ is $2$, by 
\cite[Theorem 3.1]{CMN06}, 
the $\IZ$-rank of $TF(G)$ coincides with the number of $G$-conjugacy 
classes of Klein-four 
subgroups of $G$, and this number is $2$ by 
\cite[Proposition 1.48(iv)]{Go83}.
Then by \cite[\S4]{AC}, we have 
$$TF(G)\cong\left<[\Omega^1(k_G)], [M]\right>\cong {\mathbb{Z}}^2$$ 
where $\Omega^1(k_G)$ denotes the first syzygy 
module of the trivial $kG$-module $k_G$, 
and $M$ is an indecomposable direct summand of the heart of the 
projective cover of $k_G$. (Note that there are two such 
direct summands and $M$ can be chosen to be any of them.)
\par
As a consequence, in this article we focus our attention on the determination 
of the torsion subgroup 
$TT(G)$ of  $T(G)$. We recall  that the group $X(G)$ of one-dimensional 
$kG$-modules endowed with the 
tensor product $\otimes_k$ always identifies with a subgroup of~$TT(G)$.
Our main result about the structure of $TT(G)$ in the dihedral case 
is the following.
\enlargethispage{1cm}
\begin{thm}\label{Main}
Assume that  $G$ is a finite group with  a dihedral
Sylow $2$-subgroup of order at least $8$, and let $T(G)$
be the abelian group of endo-trivial $kG$-modules
over  an algebraically closed field $k$ of characteristic $2$.
Set $\overline{G}:=G/O_{2'}(G)$.
Then the following hold:
   \begin{enumerate}
     \item[{\rm{(a)}}] If $\overline{G}\ncong \fA_6$, then $TT(G)=X(G)$.
     \item[{\rm{(b)}}] If $\overline{G}\cong \fA_6$, then either
        \begin{itemize}
  \item[(i)]  $TT(G)=X(G)$, or 
  \item[(ii)]  if there exists an indecomposable endo-trivial 
$kG$-module $V$ such that 
$[V]\in TT(G)\setminus X(G)$, then 
   $\dim_k V=9$, $V$ is simple, and $TT(G)/X(G)$ 
is an elementary abelian $3$-group.
\end{itemize}
\end{enumerate}
\end{thm}
\begin{rem}\label{rem:3A6}
Case (i) of Theorem~\ref{Main}(b) happens for example 
for $G=\fA_6$ (see \cite[Theorem 1.2]{CMN09}), whereas Case (ii) 
happens for example for $G=3.\fA_6$, the triple cover of $\fA_6$ 
(see Lemma~\ref{lem:3A6&3A7}). Moreover the central product 
$C_9\ast 3.\fA_6$ provides an example where there exist classes 
$[V]\in TT(G)$ such that $[V^{\otimes 3}]\in X(G)\!\setminus\! \{[k_G]\}$ 
(see Example~\ref{ex:CentralProduct}).\par
Furthermore, in the situation of Theorem~\ref{Main}(b)(ii), 
any  simple torsion endo-trivial $kG$-module $V$ of dimension~$9$ originates 
from the triple cover $3.\fA_6$ of the alternating group 
$\fA_6$ of degree~$6$ in the following way. 
By \cite[Theorem]{Ro11}, $E:=E(G/\ker(V))$ 
(the central product of all components of 
$G/\ker(V)$, see \cite[Definition 6.6.8]{Su86})
is quasi-simple and ${V}{\downarrow}^{G/\ker(V)}_E$ remains simple endo-trivial. 
Therefore, we must have $E=3.\fA_6$ since $3.\fA_6$ is the unique 
quasi-simple group with a 9-dimensional simple endo-trivial 
module in characteristic two by \cite[Proposition 3.8 and \S 4]{LMS13}. 
\end{rem}

\begin{cor}
If $G$ is a finite group with a dihedral Sylow $2$-subgroup of order 
at least~8 and $k$ is an algebraically closed field of characteristic $2$, 
then any indecomposable torsion endo-trivial
$kG$-module is simple, and hence  lifts uniquely  
to an ordinary irreducible 
character of~$G$.
\end{cor}

\begin{proof} This follows immediately from
Theorem \ref{Main}, the fact that $TT(G)$ consists only of classes 
of trivial source modules (see Lemma~\ref{lem:TT(G)}), and the 
fact that trivial source modules lift uniquely
(see \cite[Theorem 4.8.9(iii)]{NT89}).
\end{proof}

Our new method, developed  to treat the dihedral case of order 
at least~8, also allows us to finish off 
the classification of torsion endo-trivial modules for finite groups 
with Klein-four Sylow $2$-subgroups, 
which we started in \cite{KL14}. We note that our results in this 
case are explicit, whereas those recently obtained by 
Carlson and Th\'{e}venaz in \cite{CaTh15} 
(where they treat the general question of 
computing the group $K(G)$ for finite groups $G$ with abelian 
Sylow $p$-subgroups) only provides an algorithmic method 
to identify the group $K(G)$. \par  
\begin{thm}\label{MainTheoremK}
Assume that $G$ is a finite group with a Klein-four Sylow $2$-subgroup $P$. 
Further, set
$TT_0(G):=\{ [V]\in TT(G)\,|\,
V \text{ indecomposable and }V\in B_{0}(G) \}$ 
where $B_0(G)$ is the principal block of $G$
and $\overline G:=G/O_{2'}(G)$.
Then $TT_0(G)\cong \IZ/3\IZ$,  any indecomposable $kG$-module $V$ 
with $[V]\in TT(G)$ lifts uniquely to an $\mathcal{O}G$-lattice  
$\widehat{V}$ affording an ordinary irreducible  character 
$\chi_{\widehat{V}}\in \Irr(G)$, and the structure of $TT(G)$ is as follows:
  \begin{enumerate}[\rm(a)]
   \item
If $\overline{G}\cong P$, then $TT(G)=X(G)$.
   \item
If $\overline{G}\cong \fA_4$, then $TT(G)=X(G)$.
    \item
If $\overline{G}\cong \fA_5$, then 
$$TT(G)\cong TT(G_0)=X(G_0)\,,$$
where $G_0$ is a strongly $2$-embedded subgroup in $G$ with
$G_0/O_{2'}(G_0)\cong\mathfrak A_4$.
Furthermore,  if $V$ is a non-trivial  indecomposable endo-trivial 
$kG$-module such that $[V]\in TT_0(G)$, then $\dim_k(V)=5$.

   \item
If $\overline{G}\cong \PSL(2,q)\rtimes C_f$, where $q>5$ is a power of 
an odd prime such that $q\equiv \pm 3\pmod{8}$ and $f$ is an odd integer, 
then  
$$TT(G)\cong X(G)\oplus TT_0(G)\cong X(G)\oplus \mathbb Z/3\mathbb Z\,.$$
Furthermore,  if $V$ is a non-trivial  indecomposable endo-trivial $
kG$-module such that $[V]\in TT_0(G)$, then $\dim_k(V)=(q-1)/2$ when 
$q\equiv 3\pmod{8}$ and $\dim_k(V)=q$ when $q\equiv 5\pmod{8}$.
 \end{enumerate}
\end{thm}

The first main tool used in our investigation is Gorenstein-Walter's 
classification of finite groups $G$ with dihedral Sylow $2$-subgroups 
modulo $O_{2'}(G)$. Moreover our methods heavily 
rely on a Theorem of Schur's \cite[Theorem 3.5.8]{NT89} combined with two 
results of Navarro-Robinson \cite{NR12}, the first of which states 
that if an endo-trivial module is induced from a proper subgroup, 
then this subgroup must be strongly $p$-embedded in $G$, and the 
second of which states that simple endo-trivial modules over 
$p$-nilpotent groups of $p$-rank at least $2$ have dimension one.
This enables us  to reduce our computation of $TT(G)$ to 
that of $TT(G/O_{2'}(G))$ using Theorem~\ref{thm:PreMain}. Finally, we 
note that our methods require to decompose torsion endo-trivial modules as 
tensor products of modules over non-proper twisted group algebras, although 
the final statements of Theorem~\ref{Main} and Theorem~\ref{MainTheoremK} 
do not reflect this fact.
\par
The paper is organized as follows. In~\S 2  we  introduce the notation, 
and in~\S 3 preliminary known results on endo-trivial modules, which 
we will rely on.
In~\S 4 we prove Theorem~\ref{thm:PreMain}. In~\S 5 we give general 
results on finite groups with 
dihedral Sylow $2$-subgroups of order at least~8 and their 
endo-trivial modules, and \S\S 6--7 
are devoted to the proof of Theorem~\ref{Main}. Finally in \S 8 
we prove  Theorem~\ref{MainTheoremK}.


\section{Notation} \label{sec:not}

Throughout, unless otherwise specified we use the following notation 
and conventions. We 
let $p$ denote a prime number and $G$ a finite group of order divisible by $p$.
We assume that $(K, \mathcal O, k)$ is a
splitting $p$-modular system for all subgroups of $G$, that is, 
$\mathcal O$ is a complete discrete valuation ring of
rank one such that its quotient field $K$ is
of characteristic zero, its residue field
$k:=\mathcal O/\mathrm{rad}(\mathcal O)$ is of
characteristic $p$, and both $K$ and $k$ are
splitting fields for all subgroups of~$G$.
Modules are finitely generated left modules.
By an $\mathcal OG$-lattice, we mean an $\mathcal OG$-module 
which is $\mathcal O$-free of finite rank.
For a ring $R$, we denote by $R^\times$ the group of units of $R$.
We write ${\mathrm{Syl}}_p(G)$ for the set of all Sylow $p$-subgroups of $G$.
For a $p$-subgroup $Q$ of $G$ and $H\leq G$ with $H\geq N_G(Q)$,
we denote by $f=f_{(G,\, Q,\, H)}$ the Green correspondence 
with respect to $(G, Q, H)$, see \cite[p.276]{NT89}.
For a positive integer $n$, we denote by $C_n$  the cyclic group
of order $n$, and by $\mathfrak A_n$ 
the alternating group of degree $n$.
We write $Z(G)$ for the center of $G$, $[G,G]$ for the commutator
subgroup of $G$. If $H$ is a normal subgroup of $G$, 
and $L$ a subgroup of $G$, 
then we write $G = H \rtimes L$ if $G$ is a semi-direct product 
of $H$  by $L$.
For two $kG$-modules $M$ and $M'$,  $M\otimes_k M'$ is the tensor product 
over $k$, $M^{\otimes n}$ is the tensor product $M\otimes_k\cdots\otimes_k M$ 
of $n$ copies of $M$, we write $M^*$ for the $k$-dual of $M$, that is 
$M^* := \mathrm{Hom}_{kG}(M,k)$, 
and we write $M'\,|\, M$ when $M'$ is (isomorphic to)
a direct summand of $M$. 
We denote by $k_G$ the trivial $kG$-module.
We write $H \leq G$ if $H$ is a subgroup of $G$.
In such a case, for a $kG$-module $M$ and a $kH$-module $L$, 
we denote by $M{\downarrow}_H$ and $L{\uparrow}^G$,
respectively, the restriction of $M$ to $H$ and
the induction of $L$ to $G$.\par

We denote the Schur multiplier of $G$ by $M(G):={\mathrm{H}}^2( G,\mathbb C^\times)$.
For a $2$-cocycle $\alpha\in{\mathrm{Z}}^2(G,k^\times)$, 
we denote by $[\alpha]\in{\mathrm{H}}^2(G,k^{\times})$ the cohomology class of $\alpha$, and 
 by $k^\alpha G$ the twisted group algebra of $G$ over $k$ with respect to $\alpha$.
Then for a $k^\alpha G$-module $M$  and $H\leq G$
we write $M{\downarrow}_H$ or 
$M{\downarrow}^{k^\alpha G}_{k^\alpha H}$ for the
restriction of $M$ from $G$ to $H$.
Assume that $N\triangleleft G$. For a $2$-cocycle $\overline\alpha\in{\mathrm{Z}}^2(G/N, k^{\times})$
we denote by ${\mathrm{Inf}}_{G/N}^G (\overline\alpha)\in{\mathrm{Z}}^2(G, k^{\times})$ the inflation of
$\overline\alpha$ from $G/N$ to $G$, and by $\Inf_{k^{\overline{\alpha}}(G/N)}^{k^{\alpha}G}(M)$, the inflation of a $k^{\overline{\alpha}}(G/N)$ module $M$ to a $k^{\alpha}G$-module with $\alpha = \Inf_{G/N}^G (\overline\alpha)$. For a $k(G/N)$-module $M$
we write simply ${\mathrm{Inf}}_{G/N}^G(M)$ for the inflation
of $M$ from $G/N$ to $G$.
We denote by $B_0(G)$, the principal block of $G$, and by 
$\mathrm{Irr}(G)$  the set of
all irreducible ordinary characters of $G$.
For a $p$-block $B$ of $G$, we also denote  by
$\mathrm{Irr}(B)$ the set of all characters in $\mathrm{Irr}(G)$
which belong to $B$.
We write $1_G$ for the trivial ordinary or Brauer character of $G$.\par
We say that a $kG$-module $M$ is a {\it trivial source} module if it is a
direct sum of indecomposable $kG$-modules, all of 
whose sources are trivial modules, 
see \cite[p.218]{Thevenaz}. 
It is known that a trivial source $kG$-module $M$ 
lifts uniquely to a trivial source $\mathcal OG$-lattice, which we 
denote by $\widehat M$, see 
\cite[Theorem 4.8.9(iii)]{NT89}. 
Then,  we denote by $\chi_{\widehat M}$ the
ordinary character of $G$ afforded by $\widehat M$.
For a non-negative integer $m$ and a positive integer $n$,
we write $n_p = p^m$ if $p^m | n$ and $p^{m+1}{\not|}n$.
\par
For further standard notation and terminology,
we refer the reader to the books \cite{NT89, Thevenaz}.


\section{Preliminary results} \label{sec:prelim}
\subsection{Endo-trivial modules}\label{ssec:et}
A $kG$-module $V$ is called \emph{endo-trivial} provided 
$$\End_k(V) \cong V^* \otimes_k V \cong k_G \oplus \proj $$
as  $kG$-modules where $\proj$ denotes a projective 
direct summand (possibly the zero module).\par
Any endo-trivial $kG$-module $V$ splits as the direct sum 
$V= V_{0}\oplus \proj$ where $V_{0}:=\Omega^0(V)$, the projective-free 
part of $V$, is indecomposable and endo-trivial.
The relation $U\sim V\Leftrightarrow U_{0}\cong V_{0}$ 
is an equivalence relation on the class of endo-trivial $kG$-modules, 
and we let $T(G)$ denote the resulting set of equivalence classes 
(which we denote by square brackets). Then $T(G)$, endowed with the law 
$[U]+[V]:=[U\otimes_{k}V]$, is an abelian group called 
the \emph{group of endo-trivial modules of $G$}. 
The zero element is the class $[k_G]$ and $-[V]=[V^{*}]$.\par
Notice that if  $p\nmid |G|$, then any $kG$-module is endo-trivial, 
but the above construction of the group $T(G)$ is not valid any more.\par
The group $T(G)$ is known to be a finitely generated abelian group, 
see e.g. \cite[Corollary 2.5]{CMN06}. Therefore, we may write
$T(G)=TT(G)\oplus TF(G)$, where $TT(G)$ is the torsion subgroup of 
$T(G)$ (hence a finite group) and $TF(G)$ is a torsion-free complement.\\

We let $X(G)$ denote the group of one-dimensional $kG$-modules  
endowed with the tensor product $\otimes_{k}$, and recall that  
$X(G)\cong (G/[G,G])_{p'}$. Then by identifying a one-dimensional 
module with its class in $T(G)$, we consider $X(G)$ as a 
subgroup of $T(G)$.\par
Furthermore, define
$T_{0}(G):=\{[V]\in T(G)\,|\, V
\text{ indecomposable and }V\in B_{0}(G)\}$. 
Then $T_{0}(G)\leq T(G)$ by \cite[Proposition 9.1]{CMT14}. 
Denote by $TT_{0}(G)$ 
the torsion subgroup of $T_{0}(G)$, and by $X_{0}(G)$ the set of 
one-dimensional $kG$-modules belonging to $B_{0}(G)$. Clearly 
$X_{0}(G)\leq TT_{0}(G)\leq TT(G)$. 

\begin{rem}\label{rem:dim}
Because the dimension of projective $kG$-modules is divisible by $|G|_p$, 
if $M$ is an endo-trivial $kG$-module, then
$\dim_{k}(V)\equiv \pm 1 \pmod{|G|_{p}}$ if 
$p$ is odd; and 
$\dim_{k}(V)\equiv \pm 1 \pmod{\frac{1}{2}|G|_{2}}$ 
if $p=2$. Moreover if $V$ is indecomposable with trivial 
source, then $\dim_{k}(V)\equiv 1 \pmod{|G|_{p}}$. 
In particular, indecomposable endo-trivial $kG$-modules have 
the Sylow $p$-subgroups of $G$  as their vertices, 
and hence lie in $p$-blocks of full defect.
\end{rem}

\begin{lem}\label{lem:resinf}
Let $H$ be a subgroup of $G$, and let $P$ be a Sylow $p$-subgroup of $G$.
\begin{enumerate}[\hspace{6mm}\rm(a)]
  \item\label{lem:resinf3} If $V$ is an endo-trivial $kG$-module, 
then $V\resgh{}{H}$  is endo-trivial. 
Moreover, if $H\geq P$, then $V$ is endo-trivial if and only if 
  $V\resgh{}{H}$ is endo-trivial.
  \item\label{lem:resinf1} If  $p\,|\,|H|$, then restriction induces 
a group homomorphism 
$$\Res^{G}_{H}:T(G)\longrightarrow T(H):[V]\mapsto[V\resgh{}{H}]\,.$$
If, moreover, $H\geq N_G(P)$, then $\Res^{G}_{H}$ is injective, and for 
any  $[V]\in T(G)$ with $V$ indecomposable, $\Res^{G}_{H}([V])=[f_H(V)]$, 
where $f_H:= f_{(G,\, P,\, H)}$. 
  \item \label{lem:resinf2} If $H\triangleleft G$ such that  $p\,{\not|}\,|H|$ 
and $V$ is a $k(G/H)$-module, then $\Inf_{G/H}^{G}(V)$ is endo-trivial 
if and only if 
  $V$ is endo-trivial. Moreover the inflation from  $G/H$ to $G$
   induces an injective group homomorphism
$$\Inf_{G/H}^{G}:T(G/H)\longrightarrow T(G):[V]\mapsto[\Inf_{G/H}^{G}(V)]\,.$$
In particular, we may consider $TT(G/H)\leq TT(G)$.
\end{enumerate}  
\end{lem}

\begin{proof} 
Parts (\ref{lem:resinf3}) and (\ref{lem:resinf1}) 
are given by \cite[Proposition 2.6]{CMN06},  and 
part (\ref{lem:resinf2}) is given by
\cite[Lemma~3.2(1)]{LMz15}.
\end{proof}

Let $P$ be a Sylow $p$-subgroup of $G$. We fix the notation 
$$K(G):=\ker\left(\Res^G_P:T(G)\lra T(P)\right)\,.$$
In fact, in most cases, the torsion subgroup of $T(G)$ is equal to $K(G)$,
and has the following characterizations, which  we will use throughout.

\begin{lem}\label{lem:TT(G)}
Let $P$ be a Sylow $p$-subgroup of $G$.
\begin{enumerate}
  \item[\rm(a)] If for any $x\in G$, 
$P\cap x^{-1}Px$ is non-trivial, then $K(G)=X(G)$. 
In particular $K(N_G(P))=X(N_G(P))$.
  \item[\rm(b)]
The group $K(G)$ is exactly the set of classes of indecomposable 
trivial source endo-trivial $kG$-modules, and 
$$K(G) =  \{\, [V]\in T(G)  \,|\, \exists\text{ a module }M
\in X(N_G(P))\text{ with }V_0=f^{-1}(M) \}\, ,$$
where $f:=f_{(G,\, P,\, N_G(P))}$.
  In particular, $K(G)\leq TT(G)$ and  we may consider 
$K(G)$ as a subgroup of $K(N_G(P))=X(N_G(P))$ via the 
injective homomorphism $\Res^G_{N_G(P)}\,$. 
  \item[\rm(c)]
Furthermore, 
provided $P$ is neither cyclic, nor semi-dihedral, 
nor generalized qua\-ternion, then $K(G)=TT(G)$. 
\end{enumerate}
\end{lem}

\begin{proof} 
(a) This follows from \cite[Lemma~2.6]{MT07}. \par
(b) By definition $K(G)$ 
consists of the classes of indecomposable trivial source endo-trivial 
$kG$-modules. Hence the first claim is straightforward from (a) 
together with Lemma~\ref{lem:resinf}(b). 
Next the number of 
isomorphism classes of indecomposable trivial source $kG$-modules 
with vertex $P$ is finite, hence $K(G)$ is a finite group, 
so that we must have 
$K(G)\leq TT(G)$. The last claim follows  from Lemma~\ref{lem:resinf}(b). 
\par
(c) The claim is given by  \cite[Lemma 2.3]{CMT11}. 
\end{proof}

Our objective in this article is to consider groups with dihedral 
Sylow $2$-subgroups only, therefore in this case Lemma~\ref{lem:TT(G)}(c) 
allows us to identify $TT(G)$ with $K(G)$.\\

Finally in order to detect whether a trivial source module is endo-trivial, 
we have the following character-theoretic criterion.

\begin{thm}[{}{\cite[Theorem 2.2]{LM15}}]  
\label{thm:torchar}
Let $V$ be an indecomposable trivial source $kG$-module.
Then $V$ is endo-trivial if and only if
$\chi_{\widehat V}(u)=1$ for any non-trivial $p$-element $u\in G$.
\end{thm}

\subsection{Strongly $p$-embedded subgroups}
Recall that a subgroup $H$ of $G$ is said to be 
\emph{strongly $p$-embedded in $G$}
if $H\lneqq G$, $p{\,|\,}|H|\,$ and  
$p\,{\not|}\, |H\cap x^{-1}Hx|$ for any $x\in G\setminus H$.  
Note that any strongly $p$-embedded subgroup of $G$ 
contains the normalizer in $G$ of a Sylow $p$-subgroup. 
Moreover the operations of  induction and restriction induce equivalences of 
the stable module categories $\stmod(kH)$ and $\stmod(kG)$
if $H$ is a strongly $p$-embedded subgroup of $G$.

\begin{lem}[{}{\cite[Lemma 2.7(2)]{MT07}}]\label{lem:spe}
Let $H$ be a strongly $p$-embedded subgroup of $G$,
and let $P$ be a Sylow $p$-subgroup of $H$. 
Then $\Res^G_H:T(G)\lra T(H)$ is an isomorphism. 
Moreover, the inverse map is induced by induction, and, 
more precisely, on the indecomposable endo-trivial 
modules by the Green correspondence $f_H:=f_{(G,\, P,\, H)}$, that is
\begin{equation*}
\begin{split}
T(G)  &  =  \{\, [W\indhg{}{G}]  \,|\, [W]\in T(H)  \}  \\
         &  =  \{\, [f_H^{-1}(W)]  \,|\, W\text{ is an indecomposable 
endo-trivial }kH\text{-module} \}\,.
\end{split}
\end{equation*}
In particular $K(G)\cong K(H)$.
\end{lem}

The following result due to G.~Navarro and G.R.~Robinson
is essential for our purpose because the structure of finite groups which
have strongly $p$-embedded subgroups is in a sense very
restricted.

\begin{lem}\label{et_se}
Let $H$ be a proper subgroup of $G$, and assume that $V$ and $W$
are $kG$- and $kH$-modules, respectively, with $V=W{\uparrow}^G$.
Then the following are equivalent:
\begin{enumerate}
   \item[\rm{(1)}] $V$ is endo-trivial.
   \item[\rm{(2)}] $W$ is endo-trivial and $H$ is a strongly $p$-embedded 
subgroup of $G$.
\end{enumerate}
\end{lem}

\begin{proof}
The implication (1) implies (2) is given by [29, Lemma 1(iv)]. 
The converse is straightforward from Lemma~\ref{lem:spe}.
\end{proof}


\section{Recovering $K(G)$ via inflation from a normal p'-subgroup} 
\label{sec:K(G)viaInflation}
Before starting our investigation of endo-trivial modules over finite 
groups with dihedral Sylow 2-subgroups, we develop a general method 
enabling us to recover the subgroup $K(G)$ of $T(G)$ using 
inflation from a quotient  by a normal $p'$-subgroup of $G$.

In order to set up the technical notation for this section, 
we start by recalling and slightly generalizing well-known results of Schur.

\begin{lem}[Schur]\label{Schur}
Let $F$ be an algebraically closed field of arbitrary characteristic, 
let $H \triangleleft \, G$ and set $\overline{G}:=G/H$. Let $Y$ be 
an $n$-dimensional simple $FH$-module which is $G$-invariant. 
Then the following hold:
\begin{enumerate}[\rm(a)]
\item  $Y$ extends to an $F^{\alpha}G$-module $\widehat Y$, 
where $\alpha\in {\mathrm{Z}}^2(G,k^{\times})$, which satisfies the following 
two conditions: For any  $h\in H$, any $g\in G$ and any $y\in \widehat Y$,
\begin{enumerate}[\rm (i)]
  \item $(hg)\cdot y =h\cdot(g\cdot y)$, and 
  \item $(gh)\cdot y= g\cdot (h\cdot y)$. 
\end{enumerate}
Moreover, $\alpha(hg,h'g')=\alpha(g,g')$ for all $h,h'\in H$ 
and all $g,g'\in G$, so that  
$\alpha$ defines a $2$-cocycle 
$\overline{\alpha}:\overline{G}\times\overline{G} 
\lra k^{\times}: (gH,g'H)\mapsto \alpha(g,g')$, i.e.  
$\alpha=\Inf_{\overline{G}}^{G}(\overline{\alpha})$, 
and we have $[\overline{\alpha}^{n|H|}]=1\in 
{\mathrm{H}}^2(\overline{G}, k^{\times})$. 
\smallskip 

\item Assume that  $\widehat Y$ is an $F^{\alpha}G$-module extending $Y$, 
defined by a $2$-cocycle $\alpha\in {\mathrm{Z}}^2(G,k^{\times})$  as in {\rm{(a)}}. 
If $X$ is an $FG$-module such that
$$
X{\downarrow}_H \cong Y \oplus\cdots\oplus Y,
$$
the direct sum of $e\geq 1$ copies of $Y$, 
then there exists an $F^{ {\overline\alpha}^{-1}}\overline{G}$-module $Z$ 
such that, as $FG$-modules, 
$$
X \cong \widehat Y \otimes_F 
\Inf_{F^{ {\overline\alpha}^{-1}}\overline{G}}^{F^{\alpha^{-1}}G}(Z)\,.
$$
\end{enumerate}
\end{lem}

\begin{proof} 
Part (a) is exactly Schur's result \cite[Theorem 3.5.7]{NT89}. 
Part (b) is a generalization of Schur's theorem \cite[Theorem 3.5.8(i)]{NT89}.
More specifically, although \cite[Theorem 3.5.8(i)]{NT89} is 
stated for a module $X$ which is \textit{simple}, 
its proof only requires the assumption that 
$X{\downarrow}_H\cong Y\oplus\cdots\oplus Y$ 
in order to obtain the conclusion
that there exists an $F^{ {\overline\alpha}^{-1}}(\overline{G})$-module $Z$ 
such that $X \cong \widehat Y \otimes_F 
\Inf_{F^{ {\overline\alpha}^{-1}}\overline{G}}^{F^{\alpha^{-1}}G}(Z)$.
\end{proof}

\begin{rem}{\vbox{\ }}
\begin{enumerate}[\rm(a)]
\item We recall that inflation of $2$-cocycles 
${\mathrm{Z}}^2(\overline{G},k^{\times})\lra {\mathrm{Z}}^2(G,k^{\times}): \alpha\mapsto 
\Inf_{\overline{G}}^{G}(\overline{\alpha})$ 
induces an inflation homomorphism  $\Inf_{\overline{G}}^{G}: 
{\mathrm{H}}^2(\overline{G},k^{\times})\lra {
\mathrm{H}}^2(G,k^{\times}):[\overline{\alpha}]\mapsto 
[\Inf_{\overline{G}}^{G}(\overline{\alpha})]$ in cohomology, 
but the latter need not be injective in general. Therefore, 
it may happen that 
$\Inf_{F^{{\overline\alpha}^{-1}}\overline{G}}^{F^{\alpha^{-1}}G}(Z)$ 
is in fact a module over the non-twisted group algebra $FG$, 
while $Z$ is a module over the twisted group algebra 
$F^{ {\overline\alpha}^{-1}}\overline{G}\ncong F\overline{G}$.

\item In fact, more general statements than Lemma~\ref{Schur}(b) 
can be discussed
by making use of results of E.C. Dade. We refer the reader to the book
of A. Marcus \cite[\S 2.3.B]{Ma99}.
\end{enumerate}
\end{rem}

\begin{lem}\label{GInvariantGeneral}
Assume that $G$ has no strongly $p$-embedded subgroups and  
$V$ is an indecomposable endo-trivial $kG$-module.
If $H \vartriangleleft G$ such that $p\,{\not|}\,|H|$ and 
$L$ is a composition factor of $V{\downarrow}_H$, 
then $L$ is $G$-invariant.
\end{lem}

\begin{proof}
Set $\widetilde G:=T_G(L)$, the stabilizer of $L$ in $G$,
and let $B$ be the block of $kG$ to which $V$ belongs.
Let $\widetilde B$
be the block of $k\widetilde G$ such that 
$\widetilde B$ is the
Fong-Reynolds correspondent of $B$, see \cite[Theorem 5.5.10]{NT89}.
Then, $B$ and $\widetilde B$ are Morita equivalent. 
Write $\widetilde V$ 
for the $k\widetilde G$-module in $\widetilde B$  corresponding 
to $V$ via this Morita equivalence. Then, $V=\widetilde V{\uparrow}^G$.
Thus Lemma \ref{et_se} yields that
$\widetilde V$ is endo-trivial.
If $\widetilde G\, {\not=}\,G$, then it follows from 
Lemma~\ref{et_se}
that $\widetilde G$ is strongly $p$-embedded in $G$,
which is a contradiction.
\end{proof}

\begin{thm}\label{EndoTrivial}
Assume that the $p$-rank of $G$ is at least $2$ and 
that $G$ has no strongly $p$-embedded subgroups.
Let $H \vartriangleleft G$ with $p\,{\not|}\,|H|$,
and set $\overline G:=G/H$. Let $V$ be an indecomposable endo-trivial 
$kG$-module. Then the following hold:
\begin{enumerate}
 \renewcommand{\labelenumi}{\rm{(\alph{enumi})}}
  \item
There exists a $2$-cocycle  
$\overline\alpha\in{\mathrm{Z}}^2(\overline G,k^\times)$
such that
$$
V \ \cong \ 1b \otimes_k W,
$$
where $1b$ is a one-dimensional $k^{\alpha}G$-module for $\alpha:=\Inf_{\overline{G}}^{G}(\overline{\alpha})$ and 
$W:=\Inf_{k^{{\overline\alpha}^{-1}}\overline{G}}^{k^{\alpha^{-1}}G}(Z)$ 
for a $k^{{\overline\alpha}^{-1}}\overline G$-module $Z$.
Moreover, if $P\in\Syl_p(G)$ and $\overline{P}:=HP/H=(H\rtimes P)/H$, 
then we have $[\alpha{\downarrow}_{H\rtimes P}]=1\in
{\mathrm{H}}^2(H\rtimes P,k^\times)$ and  
$[\overline\alpha{\downarrow}_{\overline P}]=1
\in{\mathrm{H}}^2(\overline P,k^\times)$. 
\smallskip 
  \item 
Keep the notation of (a), and assume moreover that  $[V]\in K(G)$. Set $n:=|[\overline\alpha]|$,  the order of 
$[\overline\alpha]\in{\mathrm{H}}^2(\overline G,k^\times)$. 
Then $1c:= (1b)^{\otimes n}$ is a one-dimensional $kG$-module, $Z^{\otimes n}$ is a trivial source $k\overline G$-module, and 
 $W^{\otimes n}=\Inf_{\overline{G}}^{G}(Z^{\otimes n})$ is a a trivial source endo-trivial $kG$-module.
In other words,  
$$[V^{\otimes n}]=[1c]+[W^{\otimes n}]
\in X(G)+\Inf_{\overline{G}}^{G}(K(\overline{G}))\,.$$
\end{enumerate}
\end{thm}

\begin{proof}
(a) Let $L$ be a composition factor of $V{\downarrow}_H$. 
Then, by Lemma~\ref{GInvariantGeneral}, $L$ is $G$-invariant. 
Let $B$ and $b$ be the blocks of $kG$ and $kH$, respectively,
to which $V$ and $L$ belong. So clearly $B$ covers $b$.
Denote by $\theta\in{\mathrm{Irr}}(b)$ the ordinary irreducible 
character corresponding to $L$, 
and hence ${\mathrm{Irr}}(b)=\{\theta \}$.
\par
Since $B$ covers $b$, ${\mathrm{Irr}}(b)=\{\theta \}$, $\theta$
is $G$-invariant and $H$ is a $p'$-group, 
it follows from \cite[Lemma 5.5.8(ii)]{NT89} that
$$
       V{\downarrow}_H \cong L\oplus\cdots\oplus L\,.
$$
Now, since $L$ is a $G$-invariant simple $kH$-module, we know 
from another Lemma~\ref{Schur}(a) that there exist
a $2$-cocycle $\overline{\alpha}\in{\mathrm{Z}}^2(\overline G,k^\times)$ 
and a $k^\alpha G$-module $\widehat L$, for $\alpha=\Inf_{\overline{G}}^{G}(\overline{\alpha})$, such that 
$\widehat L{\downarrow}_H = L$. 
Then by Lemma~\ref{Schur}(b),
there exists a $k^{ {\alpha}^{-1}} G$-module 
$W$ such that
$$         V \ \cong \ \widehat L \otimes_k W \,,$$
where $W:=\Inf_{k^{{\overline\alpha}^{-1}}\overline{G}}^{k^{\alpha^{-1}}G}(Z)$
for a $k^{{\overline\alpha}^{-1}}\overline G$-module $Z$.

Then, we have  
$[\overline\alpha{\downarrow}_{\overline  P}]=1$ as an element
of ${\mathrm{H}}^2(\overline P, k^\times)$ 
by \cite[Proof of Theorem 3.5.11(ii)]{NT89},
and therefore $[\alpha{\downarrow}_{H\rtimes P}]=1$ as an element
of ${\mathrm{H}}^2(H\rtimes P, k^\times)$.
This implies that 
$$
V{\downarrow}_{H\rtimes P} 
  \cong \widehat L{\downarrow}_{k(H\rtimes P)}
   \otimes_k W{\downarrow}_{k(H\rtimes P)} \,,
$$
where all three modules are modules over the 
(genuine non-twisted) group algebra $k(H\rtimes P)$.
Then, by Lemma~\ref{lem:resinf}(a), $V{\downarrow}_{H\rtimes P}$ 
is endo-trivial.
Hence, both  $\widehat L{\downarrow}_{k(H\rtimes P)}$ 
and $W{\downarrow}_{k(H\rtimes P)}$
are endo-trivial by \cite[Lemma 1(iii)]{NR12}.
In addition, since $L$ is simple and $\widehat L$ is an extension of~$L$, 
$\widehat L{\downarrow}_{k(H\rtimes P)}$ is simple as well.
Thus, as the  the $p$-rank of $G$ is assumed to be at least~2, 
\cite[Theorem]{NR12} yields $\dim L=\dim\widehat L=1$.  So, we set $1b:=\widehat L$, and (a) follows.
\smallskip
(b)
First since $V$ is an indecomposable endo-trivial $kG$-module 
with $[V]\in K(G)$, we have $[V^{\otimes n}]\in K(G)$ as well. Then, 
since $|[\overline\alpha]|=n$,  
we have by (a) that
$$ V^{\otimes n} \cong 1c\otimes_k W^{\otimes n}\,,$$
where, by Lemma~\ref{Schur}(a), $1c:=(1b)^{\otimes n}$ is a one-dimensional (genuine non-twisted) 
$kG$-module  and  $W^{\otimes n}=\Inf_{\overline{G}}^{G}(Z)$ is a non-twisted
$kG$-module inflated from the non-twisted $k\overline{G}$-module $Z^{\otimes n}$.
Since $V^{\otimes n}$ is endo-trivial, again by \cite[Lemma 1(iii)]{NR12}, 
both $1c$ and $W^{\otimes n}$ are endo-trivial $kG$-modules, 
and thus $W^{\otimes n}$ is also endo-trivial as a $k{\overline G}$-module by 
Lemma~\ref{lem:resinf}(\ref{lem:resinf2}).
Now in $T(G)$ we have 
$$[V^{\otimes n}]=[1c]+[W^{\otimes n}]\,,$$
where $[1c]\in X(G)\leq K(G)$. Therefore it remains to prove that  
$[W^{\otimes n}]\in \Inf_{\overline{G}}^{G}(K(\overline{G}))$. 
But this is clear. 
Indeed, since $[W^{\otimes n}]=[V^{\otimes n}]-[1c]\in K(G)\,$,
$W^{\otimes n}$ must be a trivial source $kG$-module 
(possibly the direct sum of an indecomposable endo-trivial module and projective modules if $n>1$) by Lemma~\ref{lem:TT(G)}(b) 
and therefore so is the $k\overline{G}$-module $Z^{\otimes n}$, 
that is, $[W^{\otimes n}]=\Inf_{\overline{G}}^{G}([Z^{\otimes n}])\in \Inf_{\overline{G}}^{G}(K(\overline{G}))$.
\end{proof}

As a corollary we obtain Theorem~\ref{thm:PreMain} of the introduction.

\begin{proof}[Proof of Theorem~\ref{cor:KXK}]
Since ${\mathrm{H}}^2(\overline G,k^\times)=1$, the integer $n$ in 
Theorem~\ref{EndoTrivial}(b) is equal to~$1$. Hence the claim follows by 
identifying $K(\overline{G})$ 
with $\Inf_{\overline{G}}^{G}(K(\overline{G}))$.
\end{proof}


\section{Groups with dihedral Sylow 2-subgroups}
\label{ssec:d2n}
Throughout this section we assume that $P$ be a dihedral Sylow $2$-subgroup of $G$ of order at least $8$. 
Gorenstein and Walter proved in \cite{GorensteinWalter65} 
(see also  \cite[Theorem on p.462]{Go68}) 
that $\overline{G}:=G/O_{2'}(G)$ is isomorphic to either
\begin{enumerate}
   \item[\rm{(1)}] $P$,
   \item[\rm{(2)}] the alternating group $\fA_7$,  or
   \item[\rm{(3)}] a subgroup of $\PgammaL(2,q)$ containing $\PSL(2,q)$, 
where $q$ is a power of an odd prime.
\end{enumerate}

\begin{ass}\label{AssD}
For the purposes of our computations, we split case (3) above  in further 
subcases and say that $G$ satisfies the hypothesis:
\begin{enumerate}
   \item[\textbf{(D1)}] if $\overline{G}\cong P$;
   \item[\textbf{(D2)}] if $\overline{G}\cong \fA_7$;
   \item[\textbf{(D3)}] if 
$\overline{G}\cong {\mathrm{PSL}}(2,9) \cong \mathfrak A_6$; 
   \item[\textbf{(D4)}] if $\overline{G}\cong \PGL(2,9)
\cong \mathfrak A_6.2_2$;
   \item[\textbf{(D5)}] if $\overline{G}\cong \PSL(2,q)$,  
where  $q$ is a power of an odd prime with $q\neq 9$, and 
$q\equiv\pm 1\pmod{8}$;
   \item[\textbf{(D6)}] if $\overline{G}\cong \PGL(2,q)$,  
where $q$ is a power of an odd prime with $q\neq 9$;
   \item[\textbf{(D7)}]  if $\overline{G}\cong \PSL(2,q)\rtimes C_f$,  
where $q$ is a power of an odd prime with 
$q\neq9$, $q\equiv\pm 1\pmod{8}$, and $f>1$ is odd; 
   \item[\textbf{(D8)}]  if $\overline{G}\cong \PGL(2,q)\rtimes C_f$,  
where $q$ is a power of an odd prime with $q\neq 9$, and $f>1$ is odd. 
\end{enumerate}
\end{ass}
The splitting of case (3) into  \textbf{(D3)}-\textbf{(D8)}  
follows from the fact that the structure of $\PgammaL(2,q)$, 
where $q=r^m$ is a power of an odd prime~$r$, is well-known:
$\PgammaL(2,q)\cong \PGL(2,q)\rtimes 
{\mathrm{Gal}}(\mathbb F_{q}/\mathbb F_{r})$,  
where ${\mathrm{Gal}}(\mathbb F_{q}/\mathbb F_{r})$  
is cyclic of order $m$. Moreover 
\cite[Chapter~6~(8.9)]{Su86} shows that $f$ is odd.

\begin{lem}\label{BenderSuzuki}
There are no strongly $2$-embedded subgroups in $G$.
\end{lem}

\begin{proof}
This follows from the Bender-Suzuki Theorem
\cite[Satz 1]{Be71} (cf. \cite{Su64}) and also
a result of Gorenstein-Walter \cite{GorensteinWalter65}, see 
\cite[Theorem on p.462]{Go68}, too.
\end{proof}

\begin{lem}\label{Multiplier}
Set $h:=|{\mathrm{H}}^2(\overline  G, k^\times)|$. Then 
$$h= \begin{cases}
    3  & \text{ if } \, \overline  G
\in \{\mathfrak A_6, \mathfrak A_7, \PGL(2,9)\}, \\
    1  & \text{otherwise}.
\end{cases}$$
\end{lem}

\begin{proof}
Since $k$ has characteristic $2$, it follows from \cite[Proposition~3.2]{Ya64}
(see also \cite[Lemma~5]{Ko82} and
\cite[Lemma~3.5.4(ii)]{NT89}) that
${\mathrm{H}}^2(\overline  G,k^\times)\cong M(\overline  G)_{2'}$,
where $M(\overline  G)_{2'}$ is the $2'$-part of 
$M(\overline  G)={\mathrm{H}}^2(\overline  G,\mathbb C^\times)$.
\par
If \textbf{(D1)} holds, then $h=1$ by \cite[Theorem 2.7.4]{NT89}. 
If \textbf{(D2)} or \textbf{(D3)} holds, then $h=3$ by 
\cite[p.10 and p.4, respectively]{Atlas}.\par
Assume \textbf{(D4)} holds. Since  $\overline  G/\PSL(2,9)$ is 
cyclic and $\PSL(2,9)$ is perfect, we know by \cite[Theorem 3.1]{Jo74} 
that $|M(\overline  G)|\,{\Big|}\,|M(\PSL(2,9))| = 6$.  So $h=1$ or 
$h=3$. Then \cite[$A_6$ (mod 2)]{ModularAtlas} yields $h=3$.\par
Next assume \textbf{(D5)} holds.
It is known that if $q\neq 9$ is a power of an odd prime, then 
$|M(\PSL(2,q))| = 2$ by a result of R.~Steinberg
in \cite[Theorem 4.9.1(ii)]{Ka85}. Hence, we have $h=1$.
\par
Next, consider the particular case that $q=3$ when \textbf{(D6)} 
or \textbf{(D8)} holds 
(note that $q\, {\not=}\, 3$ if  \textbf{(D7)} holds).
Assume first that \textbf{(D6)} occurs.
Then, in the former case, 
$\overline G\cong{\mathrm{PGL}}(2,3)\cong\mathfrak S_4$ 
and by Schur's result \cite[Theorem 4.3.8(i)]{Ka85}  
we have that $|M(\mathfrak S_4)|=2$, 
so that $h=1$.
Assume that \textbf{(D8)} occurs.
It follows from \cite[Theorem 3.1(i)]{Jo74} that
$$
|M(\overline G)|{\Big |}
|(M(\mathfrak S_4)|\cdot
|M(\mathfrak S_4/[\mathfrak S_4, \mathfrak S_4])| = 2\times 2=4.$$
Hence  we have $h=1$.
\par
Finally assume that \textbf{(D6)}, \textbf{(D7)} or \textbf{(D8)} 
holds with $q>3$.
Then
$$
\overline G/{\mathrm{PSL}}(2,q) \cong \begin{cases}
    C_2  & \text{if \textbf{(D6)} holds}, \\
    C_f   & \text{if \textbf{(D7)} holds}, \\
    C_{2f} & \text{if \textbf{(D8)} holds},
\end{cases}
$$
which, in particular, is cyclic in all three cases. Thus, 
as $\PSL(2,q)$ is perfect, 
it follows from \cite[Theorem 3.1(i)]{Jo74} that
$|M(\overline  G)|\,{\Big|}\,|M({\mathrm{PSL}}(2,q))| = 2$. 
Therefore we obtain $h=1$.
\end{proof}


\section{Torsion endo-trivial modules in the dihedral case: 
the basic examples} 
\label{sec:ETforGbar}
We now turn to the description of $TT(G)$ 
for groups $G$ with dihedral Sylow 2-subgroup of order at least $8$.  
First we investigate the case when $O_{2'}(G)$ is trivial and prove 
that torsion endo-trivial modules are always 
one-dimensional in this case.
Throughout this section we use the notations $G$, $P$ 
and $\overline G$ as in \S 5.

\begin{prop}\label{H=1}
If $O_{2'}(G)=1$, then $TT(G)=K(G)=X(G)$.
\end{prop}

\begin{proof}
By assumption, we have $G=\overline G$, thus we may go through 
the possibilities for $G$ according to Hypotheses~\ref{AssD}.
\par
 If $\textbf{(D1)}$ holds, i.e. $G=P$, then $TT(G)=\{[k_G]\}$ 
by \cite[Theorem~5.4]{CaTh00}. 
If $\textbf{(D2)}$ holds, i.e. $G=\fA_7$, then $TT(G)=\{[k_{G}]\}$ 
by \cite[Theorem~B(a)]{CMN09}.\par
Now assume that $G$ satisfies one of $\textbf{(D3)}$ to $\textbf{(D8)}$.  
Set $N:=N_{G}(P)$. As $P$ is dihedral of order at least 8, its 
automorphism group is a $2$-group, so that 
$$N=PC_{G}(P)=P\times O_{2'}(C_{G}(P))\,.$$
In particular, if $G$ satisfies $\textbf{(D3)}$, $\textbf{(D4)}$, 
$\textbf{(D5)}$ or $\textbf{(D6)}$, 
then $N=P$, and hence  $X(N)=\{[k_N]\}$,
so that Lemma~\ref{lem:TT(G)}(b) yields $TT(G)=K(G)=\{[k_G]\}=X(G)$.
\par
Finally assume that $G$ satisfies $\textbf{(D7)}$ or $\textbf{(D8)}$.
Then $N=P\times C_f$, so that $X(N)\cong C_f$. 
But clearly  $X(G)\cong C_{f}$ as well, so that the $kG$-Green 
correspondents of the one-dimensional $kN$-modules are 
all one-dimensional. Hence  
$$TT(G)=K(G)=X(G)\cong C_{f}$$
by Lemma~\ref{lem:TT(G)}(b).
\end{proof}

As a consequence, we see that any torsion endo-trivial 
module of a finite group with dihedral Sylow $2$-subgroup 
of order at least 8 which lies in the principal block has 
to be one-dimensional.

\begin{cor}\label{cor:TT0}
There is an isomorphism of groups $TT_{0}(G)\cong X_{0}(G)$.
\end{cor}

\begin{proof}
Since $O_{2'}(G)$ acts trivially on the principal block, 
Lemma~\ref{lem:resinf}(c) yields 
$$TT_{0}(G) = \Inf_{G/O_{2'}(G)}^{G}\left( TT_{0}(G/O_{2'}(G)) \right)\,.$$ 

Now, by Proposition~\ref{H=1}, $TT_{0}(G/O_{2'}(G))=X_0(G/O_{2'}(G))$. 
The claim follows.
\end{proof}

Next we consider the triple covers of $\fA_6$, $\fA_7$ and $\PGL(2,9)$ 
whose Schur multipliers have non-trivial $2'$-parts 
as seen in Lemma~\ref{Multiplier}.
The next lemma also shows that  $TT(G)$ is not isomorphic 
to $TT(G/O_{2'}(G))$ via inflation in general.

\begin{lem}\label{lem:3A6&3A7}
\begin{enumerate}[\rm(a)]
  \item Let $H:=3.\fA_{6}$, then  
$TT(H) = \{[k_H], [9_1], [9_2]\}\cong \mathbb Z/3\mathbb Z$,
where $9_1$ and $9_2$ are mutually dual 
$9$-dimensional, faithful, simple and trivial source $kH$-modules.
  \item Let $G:=3.\fA_{7}$, then  $TT(G)= TT_0(G)=X(G)= \{[k_{G}]\}$. 
  \item Let $R:=3.{\mathrm{PGL}}(2,9)$, 
then  $TT(R)= TT_0(R)=X(R)= \{[k_{R}]\}$. 
\end{enumerate}  
\end{lem}

In the following proof, 
an ordinary irreducible character of degree $d$ of a group $G$  
is denoted by $\chi_{d}$, whereas irreducible Brauer characters are 
denoted by their degrees,  
and are identified with the corresponding simple $kG$-modules. 
Moreover  ordinary irreducible characters are labelled according 
to \cite[Decomposition Matrices]{ModularAtlas}. 
We note that the above result about $3.\fA_6$ 
and $3.\fA_7$ appears in \cite[Proposition~6.1]{LaMaz15II}, 
where it was obtained via a {\sf MAGMA} computation \cite{MAGMA},
while we give here a character-theoretic proof.

\begin{proof} First note that we may identify $H$ with a subgroup of $G$ and 
let $P\in{\mathrm{Syl}}_{2}(H)$, so that $P\in{\mathrm{Syl}}_{2}(G)$ as well. 
Then $P\cong D_{8}$ (the dihedral group of order $8$), 
$N:=N_{H}(P)=N_{G}(P)\cong C_3\times P$, 
and $N\leq H\leq G$. 
In particular, $X(N)\cong \mathbb Z/3\mathbb Z$. 
By Lemma~\ref{lem:TT(G)}(b), 
we need to determine whether the $kH$- and $kG$-Green 
correspondents of the two non-trivial elements 
$1a, 1a'\in X(N)$ are endo-trivial. 
So, set $\mathfrak f_H:= f_{(H,\, P,\, N)}$ and 
$\mathfrak f_G:= f_{(G,\, P,\, N)}$.
But $1a'=1a^{*}$, 
and endo-triviality is preserved by passage to the $k$-dual, 
hence we only need to determine whether $\mathfrak f_{H}^{-1}(1a)$ and 
$\mathfrak f_{G}^{-1}(1a)$ are endo-trivial modules.\par
Note that the group $3.\fA_{6}$ has a maximal 
subgroup $M\cong C_3\times \fS_{4}$ (indeed the product is direct 
because $M(\fS_4)\cong C_2$) containing $N$,
so that $N\leq M\leq H\leq G$. 
Hence 
$$TT(M)=X(M)\cong \mathbb Z/3\mathbb Z\,$$ and we denote 
by $1{b}$ and  $1{b}^{*}$ the two non-trivial elements of $X(M)$, 
which we identify with their ordinary characters. \par
\rm(a) 
We may assume that 
$\mathfrak f_{H}^{-1}(1a)\,\, {|}\, \,\, 1b{\uparrow}_{M}^{H}$. 
Then we calculate that 
$$ \chi_{1b}{\uparrow}_M^H = \chi_{6_{1}} + \chi_{9_{2}}\,.$$
Thus $\dim_{k}({\mathrm{End}}_{kH}(1b{\uparrow}_M^H)) = 2$
by making use of Scott's Theorem in \cite{Sc73} (see also 
\cite[II Theorem 12.4(i) and I Lemma 14.5]{Lan83}). 
Then it follows from the 2-decomposition matrix of $H$  
\cite[$A_6$(mod 2)]{ModularAtlas}
that 
$$
1b{\uparrow}_{M}^{H} = 9_2
\oplus \bigl(\begin{smallmatrix}3c \\ 3d\end{smallmatrix}\bigr)
$$
as $kH$-modules, where $3c$ and $3d$ 
are the two non-isomorphic simple modules of dimension $3$ 
belonging to the $2$-block containing ${9}_{2}$,
and   
$\bigl(\begin{smallmatrix}3c \\ 3d\end{smallmatrix}\bigr)$
is a uniserial $kH$-module with composition factors $3c$ and $3d$.
Therefore $\mathfrak f_{H}^{-1}(1b) = {9}_{2}$, 
which affords $\chi_{9_{2}}\in\Irr(H)$.
Hence $\mathfrak f_{H}^{-1}(1b)$
is endo-trivial by Theorem~\ref{thm:torchar} as it takes value $1$ 
on any non-trivial $2$-element of $H$, see \cite[p.5]{Atlas}.
As a consequence,  $TT(H)\cong X(N)\cong \mathbb Z/3\mathbb Z$.
\par
\rm(b)   
Now passing to $G$,
$\mathfrak f_{G}^{-1}(1a)\,\, {|}\,\,\, f_M^{-1}(1b){\uparrow}_{H}^{G}$,
where $f_M:= f_{(H,\, P,\, M)}$ 
(recall that, on the other hand, 
$\mathfrak f_G$ denotes the Green correspondence
with respect to $(G, P, N)$). 
We calculate that 
$$\chi_{9_{2}}{\uparrow}_H^G 
= \chi_{15_{4}} + \chi_{24_{1}} + \chi_{24_{3}} \,.$$ 
But 
$\chi_{24_{1}}, \, \chi_{24_{3}}\in\Irr(G)$ being 
defect-zero characters, we obtain that $\mathfrak f_{G}^{-1}(1a)$ 
affords $\chi_{15_{4}}\in\Irr(G)$,  
which does not take value~$1$ on the unique conjugacy 
class of involutions, see \cite[p.10]{Atlas}. 
Thus, by Theorem~\ref{thm:torchar},  
$\mathfrak f_{G}^{-1}(1a)$ and $\mathfrak f_{G}^{-1}(1a')$ are not 
endo-trivial modules, and we must have $TT(G)=\{[k_{G}]\}$.\par 
\rm(c)
By \cite[$A_6.2_2$ (mod 2)]{ModularAtlas}, 
$R$ has a unique block of full defect, namely the principal block.  
Therefore it follows from Remark~\ref{rem:dim} that $TT(R)=TT_0(R)$ 
and we conclude from Corollary~\ref{cor:TT0} and the proof of 
Proposition~\ref{H=1} that $TT(R)=X(R)=\{[k_R]\}$.
\end{proof}

\begin{cor}[See Remark~\ref{rem:3A6}.]\label{cor:3A6}
Set $H:= O_{2'}(G)$ and $\overline G:= G/H$.
Assume further that $\overline G\cong\mathfrak A_6$ and $G$ has a unique 
component $E:=E(G)\cong 3.\mathfrak A_6$. Then, the
following hold:
\begin{enumerate}[\rm(a)]
  \item $G=HE$, $[H,E]=1$ (and hence $G$ is a central product of
$H$ and $E$), and moreover $H\cap E=Z(E)\cong C_3$ and
$Z(G)\leq H=C_G(E)$.
  \item Let $W$ be a simple 
$kE$-module such that  $[W]\in K(E)$ and $\dim(W)=9$
as in Lemma~\ref{lem:3A6&3A7}(a).
Then, there exists
a simple and trivial source $kG$-module $V$ such that 
$V{\downarrow}_E \cong W\oplus\cdots\oplus W$ ($m$ summands) 
for a positive integer $m$.
Further, $V$ is endo-trivial if and only if $W$ extends to $V$.
In particular, if ${\mathrm{H}}^2(G/E,k^\times)=1$, then 
there exists a $9$-dimensional simple $kG$-module whose class is in $K(G)$.
\end{enumerate}
\end{cor}

\begin{proof}
(a) Since $E\triangleleft G$ and $G/H$ is non-abelian simple,
we have $G=HE$. Then,
$$
\mathfrak A_6\cong G/H=HE/H\cong E/(H\cap E)
=(3.\mathfrak A_6)/(H\cap E).
$$
Hence $C_3\cong H\cap E=Z(E)$. We have also $Z(G)\leq H$
since $Z(G)\triangleleft G$ and $G/H$ is non-abelian simple.
Further, if $G=HC_G(E)$, then $[P,E]=1$, a contradiction,
so that $C_G(E)\leq H$ since $C_G(E)\triangleleft G$ and
$G/H$ is non-abelian simple. 

Next, we claim that $H\leq C_G(E)$.
Take any $h\in H$, and let $\phi_h$ be an element of 
${\mathrm{Aut}}(E)$ given by $y\mapsto h^{-1}yh$
for $y\in E$. Since $|Z(E)|=3$, $\phi_h$ acts trivially on 
$Z(E)$, and hence we can consider that
$\phi_h\in{\mathrm{Aut}}(E/Z(E))={\mathrm{Aut}}(\mathfrak A_6)$.
Since $|h|$ is odd and $|{\mathrm{Out}}(\mathfrak A_6)|=4$
by \cite[p.4]{Atlas}, we know that 
$\phi_h$ is an inner automorphism of $\mathfrak A_6$, and hence
$\phi_h$ is an inner automorphism of $E$. This implies that
there is an element $y_0\in E$ with $hy_0\in C_G(E)$.
Since we know already that $C_G(E)\leq H$,
we have $y_0\in H$, so that $y_0\in H\cap E=Z(E)\leq C_G(E)$.
Therefore $h\in C_G(E)$.
\par
(b) The first part follows  easily from (a) and the Clifford Theorem. 
Then, the second part follows from Theorem~\ref{thm:torchar}. 
The final part follows immediately.
\end{proof}


\section{Torsion endo-trivial modules in the dihedral case: 
Proof of Theorem~\ref{Main}}

We now turn to the general case and prove 
Theorem~\ref{Main} of the introduction.
Throughout this section we use the notations $G$, $P$ and $\overline G$
as in \S 5, and further set $H:=O_{2'}(G)$, and hence $\overline G:=G/H$ 
and $\overline P:=HP/H\cong P$.

\begin{prop}\label{FirstKeyLemma}
Let $V$ be an indecomposable endo-trivial $kG$-module such 
that $[V]\in K(G)$. Then the following hold:
 \begin{enumerate}
    \item[{\rm{(a)}}] If $\overline{G}\ncong \fA_6$, then $\dim_k V=1$.
        \item[{\rm{(b)}}] If $\overline{G}\cong \fA_6$ and 
$\dim_k V\neq 1$, then $\dim_k V=9$, $V$ is simple,
$[V^{\otimes 3}]\in X(G)$.
\end{enumerate}
\end{prop}

\begin{proof}
Since $G$ has no strongly $2$-embedded subgroups 
by Lemma~\ref{BenderSuzuki},
Theorem~\ref{EndoTrivial}(a) yields that there exist a 
$2$-cocycle $\overline\alpha\in{\mathrm{Z}}^2(\overline G,k^\times)$,
a one-dimensional $k^{\alpha}G$-module $1b$  for  $\alpha:={\mathrm{Inf}}_{\overline G}^G(\overline\alpha)$
and a $k^{ {\overline\alpha} ^{-1}}\overline G$-module $W$, 
which we may regard as a $k^{\alpha^{-1}}G$-module 
via inflation from $\overline{G}$ to $G$, such that
$$  V \ \cong \ 1b\otimes_k W \,.$$
Now, we go through the possibilities for $\overline{G}$ 
according to Hypotheses~\ref{AssD} and compute $\dim_k(V)$ in each case.\par
To start with, if $G$ satisfies one of the hypotheses 
$\textbf{(D1)}$, $\textbf{(D5)}$, $\textbf{(D6)}$, 
$\textbf{(D7)}$, or $\textbf{(D8)}$, then 
${\mathrm{H}}^2(\overline G,k^\times)=1$ by Lemma~\ref{Multiplier}. 
Therefore Theorem~\ref{cor:KXK} yields 

$$K(G)=X(G)+\Inf_{\overline{G}}^{G}(K(\overline{G}))\,.$$
Besides $K(\overline{G})=X(\overline{G})$ by Proposition~\ref{H=1}. 
In consequence  $K(G)=X(G)$ and it follows that $\dim_k(V)=1$ 
in all cases.
\par
Hence we may assume that $G$ satisfies one 
of $\textbf{(D2)}$, $\textbf{(D3)}$, or $\textbf{(D4)}$. 
If $[\overline\alpha]$ is trivial, then by the same argument 
as above we obtain $\dim_k(V)=1$. Therefore we assume 
from now on that $[\overline\alpha]$ is non-trivial and  it 
follows from Lemma~\ref{Multiplier} that $|[\overline\alpha]|=3$.  
Then there exists a non-split central extension 
$$
(Z): \ \ 
1\rightarrow \widetilde Z\rightarrow 
\widetilde G\rightarrow \overline  G\rightarrow 1\,
$$
where $\widetilde Z\cong C_3$,  and we write 
$\widetilde G = 3.{\overline  G}$,  
the triple cover of $\overline  G$. Then it follows from Theorems of Schur
\cite[p.214, and Theorems 3.5.21 and 3.5.22]{NT89} that
the module $W$ over $k^{\overline\alpha^{-1}}\overline  G$ corresponds to
an indecomposable  $k\widetilde G$-module $\widetilde W$ such that 
$\widetilde W=W$ as $k$-vector spaces,  and moreover, 
if $\widetilde P\in{\mathrm{Syl}}_2(\widetilde G)$, then 
$W{\downarrow}^{k^{ {\overline\alpha} ^{-1}}\overline G}_{k\overline P} 
\cong \widetilde W{\downarrow}_{k\widetilde P}$
as $k\widetilde P$- (and also as $k\overline  P$-) 
modules via the canonical isomorphism 
$\widetilde P\cong \overline  P\cong P$.
We claim that $[\widetilde W]\in K(\widetilde G)$.
Indeed, since $[\alpha{\downarrow}_{H\rtimes P}] = 1$
by Theorem~\ref{EndoTrivial}(a) and its proof, 
we have
$$
V{\downarrow}_{H\rtimes P} 
  = (1b){\downarrow}_{k(H\rtimes P)}
   \otimes_k W{\downarrow}_{k(H\rtimes P)}
$$
as a tensor product of (non-twisted) $k(H\rtimes P)$-modules.
By Lemma~\ref{lem:TT(G)}(b), $[V{\downarrow}_{H\rtimes P}]\in K(H\rtimes P)$ 
since $[V]\in K(G)$ and $\Res^G_{H\rtimes P}$ 
is a group homomorphism. 
Therefore, since $K(H\rtimes P)$ is a subgroup of $T(H\rtimes P)$, we have
$$
[W{\downarrow}_{k(H\rtimes P)}]
=[V{\downarrow}_{H\rtimes P} ]-[(1b){\downarrow}_{k(H\rtimes P)}]
\in K(H\rtimes P)=TT(H\rtimes P)\,,
$$
where the latter equality of groups holds by Lemma~\ref{lem:TT(G)}(c).
Thus, as $H$ acts trivially on~$W$, it follows from 
Lemma~\ref{lem:resinf}(c) that  
$$
[W{\downarrow}_{k((H\rtimes P)/H)}] \, \in \, TT((H\rtimes P)/H)=TT(\overline P)=K(\overline P)\,,
$$
where, again, the latter equality of groups holds by Lemma~\ref{lem:TT(G)}(c).
But $K(\overline{P})=\{[k_{\overline{P}}]\}$ 
by definition since $\overline{P}$ is a $2$-group. 
Thus, $[W{\downarrow}_{\overline P}]=[k_{\overline P}]$, and hence 
$[\widetilde W{\downarrow}_{\widetilde P}\,]=[k_{\widetilde P}]$, 
so that by  Lemma~\ref{lem:resinf}(a), $\widetilde W$ 
itself is a trivial source endo-trivial $k\widetilde G$-module, 
that is,
 $[\widetilde W\,] \in K(\widetilde G)$.\par
Now assume that $G$ satisfies $\textbf{(D4)}$, 
i.e. $\overline  G \cong{\mathrm{PGL}}(2,9)$. 
In this case, $\widetilde G=3.{\mathrm{PGL}}(2,9)$.
Since $\widetilde W$ is an indecomposable endo-trivial module, 
Lemma~\ref{lem:3A6&3A7}(c) yields 
$\widetilde W = k_{\widetilde  G}$. Thus 
$$\dim_k V=\dim_k(1b\otimes_k W)=\dim_k  W=\dim_k \widetilde W=1\,.$$

Next assume that $G$ satisfies $\textbf{(D2)}$, i.e. 
$\overline  G \cong \fA_7$. In this case, $\widetilde G=3.\fA_7$.
By the above, the $k\widetilde G$-module $\widetilde W$ is  
indecomposable endo-trivial 
such that  $[\widetilde W]\in K(\widetilde G)$.
Thus, Lemma~\ref{lem:3A6&3A7}(b) yields $\widetilde W=k_{\widetilde G}$, 
so that $\dim_k V=\dim_k W=\dim_k \widetilde W=1$.\par
Finally assume that $G$ satisfies $\textbf{(D3)}$, 
i.e. $\overline  G \cong \fA_6$, so that $\widetilde G=3.\fA_6$.
Since $\widetilde W$ is  indecomposable endo-trivial 
and  $[\widetilde W]\in K(\widetilde G)$, 
Lemma~\ref{lem:3A6&3A7}(a) implies that  
$\widetilde W\in\{k_{\widetilde G}, \, 9_1, \, 9_2 \}$ 
(where $9_1$ and $9_2$ are as in Lemma \ref{lem:3A6&3A7}(a)) 
and therefore is a simple module.
Hence, again, we compute 
$$
\dim_k V=\dim_k(1b\otimes_k W)
=\dim_k W=\dim_k \widetilde W\in\{1,9\}\,.
$$

Next, we claim that $V$ is a simple module as well. 
We may assume, without loss of generality, that $\widetilde W=9_1$,
which lies in one of the two non-principal 
$2$-blocks of full defect of $3.\fA_7$. 
So let $B$ and $\widetilde B$ be the blocks of $kG$ and $k\widetilde G$,
respectively, to which $V$ and $\widetilde W$ belong.
Then, $B$ and $\widetilde B$ correspond to each other by the
Morita equivalence given by the result of Morita
in \cite[Lemma 2]{Ko82} (see also \cite{Mo51} 
and \cite[Theorem 5.7.4]{NT89}), 
and $V$ and $\widetilde W$ correspond
to each other via this Morita equivalence.
Hence $V$ is a simple $kG$-module
since $\widetilde W$ is simple.\par
Now, by Theorem~\ref{EndoTrivial}(b), we have that 
$V^{\otimes 3}=(1b)^{\otimes 3}\otimes_kW^{\otimes 3}$ is 
the tensor product of a (non-twisted) $kG$-module 
$(1b)^{\otimes 3}$ and a (non-twisted) $kG$-module 
$W^{\otimes 3}$ such that  
$[W^{\otimes 3}]\in \Inf_{\overline{G}}^{G}(K(\overline{G}))$. 
Since $\overline{G}=\fA_6$, we have  
$K(\overline{G})=\{[k_{\overline{G}}]\}$ by 
Proposition~\ref{H=1}. Therefore 
$[W^{\otimes 3}]=[k_{G}]$, and hence 
$[V^{\otimes 3}]=[(1b)^{\otimes 3}]\in X(G)$.
\par
\end{proof}

\begin{proof}[Proof of Theorem~\ref{Main}]
Since the Sylow $2$-subgroups of $G$ are dihedral of order
at least $8$, 
Lemma~\ref{lem:TT(G)}(c) yields $TT(G)=K(G)$. 
Thus the claims follow from Proposition~\ref{FirstKeyLemma}.
\end{proof}

\begin{ex}\label{ex:CentralProduct}
We now give an example of a group $G$ with a dihedral Sylow 
$2$-subgroup $P$ of order $8$, 
for which there exist classes $[V]\in TT(G)$ such that 
$$[V^{\otimes 3}]\in X(G)\!\setminus\! \{[k_G]\}\,.$$
We put ourselves in the situation of Corollary \ref{cor:3A6}
and consider the following example. We define
$G$ to be the central product defined by 
$$G:=C_9\ast 3.\fA_6\cong (C_9\times 3.\fA_6)/C_3\,,$$
so that $\overline{G}:=G/O_{2'}(G)\cong\fA_6$, 
but $H^2(\overline{G},k^{\times})\cong C_3$.  
It easily follows that $X(G)\cong \IZ/3\IZ$ and  
$X(N_G(P))\cong \IZ/9\IZ$. Then, using {\sf GAP} \cite{CTblLib}, 
we compute the following. First, inducing the nine 
linear characters of $N_G(P)$ to $G$, we find that 
the $kG$-Green correspondents of the associated 
$kN_G(P)$-modules afford the three linear characters 
corresponding to $X(G)\cong \IZ/3\IZ$ (this part is obvious) and six $9$-dimensional 
ordinary characters $\chi_{9_1},\ldots,\chi_{9_6}$, reducing 
modulo $2$ to $9$-dimensional simple $kG$-modules 
$9_1,\ldots, 9_6$, respectively 
(all lying in pairwise distinct blocks). 
Then Theorem~\ref{thm:torchar} ensures that these modules 
are all endo-trivial since their ordinary characters take 
value one on any non-trivial $2$-element of $G$. Hence 
we conclude that 
$$TT(G)=K(G)\cong X(N_G(P))\cong \IZ/9\IZ.$$
Finally we see that the characters 
$\chi_{\widehat 9_i}\otimes_K \chi_{\widehat 9_i}
\otimes_K \chi_{\widehat9_i}$ 
for each $1\leq i\leq 6$ have no trivial constituents, 
so that $[(9_i)^{\otimes 3}] \in X(G)\!\setminus\! \{[k_G]\}$ 
for each $1\leq i\leq 6$.
\end{ex}


\section{Groups with Klein-four Sylow 2-subgroups revisited}
\label{ssec:k4}

The purpose of this section is to prove Theorem~\ref{MainTheoremK}.
Throughout this section we assume that $P$ is a Klein-four 
Sylow $2$-subgroup of $G$, that is $P\cong C_2\times C_2$, and let $N:=N_G(P)$.
Furthermore, we use the notations $H:=O_{2'}(G)$, $\overline G:=G/H$ and
$\overline  P:=(HP)/H=(H\rtimes P)/H\cong P$.

\begin{lem}\label{G_C2xC2}
One of the following holds:
\begin{enumerate}
   \item[{\rm{(1)}}]
$\overline G\cong P$.
   \item[{\rm{(2)}}]
$\overline G\cong
\mathfrak A_4 \cong{\mathrm{PSL}}(2,3)$,
   \item[{\rm{(3)}}] 
$\overline G\cong {\mathrm{PSL}}(2,q)\rtimes C_f$, where $q=r^m$ is a power of an odd prime $r$ 
with $3<q\equiv\pm 3$ {\rm{(mod 8)}}, $f$ is an odd integer, 
and  $C_f\leq{\mathrm{Gal}}(\mathbb F_q/\mathbb F_r)\cong C_m$.
\end{enumerate}
\end{lem}

\begin{proof}
This follows from \cite[Theorem I]{Walter1969}, see also 
\cite[Proof of Theorem 6.8.7, (8.9) in Chapter 6 and Theorem 6.8.11]{Su86}.
\end{proof}

\begin{ass}
For the purposes of our computations, we split case (3) above  
in further subcases and say that $G$ satisfies the hypothesis:

\begin{enumerate}
   \item[\textbf{(K1)}] if $\overline{G}\cong P$;
   \item[\textbf{(K2)}] if $\overline{G}\cong 
\fA_4\cong{\mathrm{PSL}}(2,3)$;
   \item[\textbf{(K3)}] if 
$\overline{G}\cong \fA_5\cong{\mathrm{PSL}}(2,5)$;
   \item[\textbf{(K4)}] if $\overline{G}\cong 
{\mathrm{PSL}}(2,q)\rtimes C_f$, where   $q=r^m$ is a power of an odd prime $r$ 
with $3<q\equiv 3$ {\rm{(mod 8)}} and $f$ is odd  with $f|m$
(that is  
$C_f\leq{\mathrm{Gal}}(\mathbb F_q/\mathbb F_r)\cong C_m$);
\item[\textbf{(K5)}] 
if $\overline{G}\cong{\mathrm{PSL}}(2,q)\rtimes C_f$,  
where   $q=r^m$ is a power of an odd prime $r$
with $5<q\equiv 5$ {\rm{(mod 8)}} for an odd $f$ with $f|m$
(that is  
$C_f\leq{\mathrm{Gal}}(\mathbb F_q/\mathbb F_r)\cong C_m$). 
\end{enumerate}
\end{ass}

\begin{lem}\label{BenderSuzuki_K}
The following two conditions are equivalent:
\begin{enumerate}
\item[{\rm{(1)}}] $G$ has a strongly $2$-embedded subgroup.
\item[{\rm{(2)}}] $\overline G\cong\mathfrak A_5$.
\end{enumerate}
Moreover, if {\rm{(1)}} holds, then any strongly $2$-embedded subgroup $G_0$
of $G$ is of the form $HN$ where 
$N:=N_G(P)$ for a suitable choice of $P\in \Syl_2(G)$,
so that $G_0/O_{2'}(G_0)\cong \mathfrak A_4$.
\end{lem}

\begin{proof}
This follows from the Bender-Suzuki Theorem 
\cite[Theorem 6.4.2(2)]{Su86}.
\end{proof}

\begin{lem}\label{Multiplier_K}
The group ${\mathrm{H}}^2(\overline  G, k^\times)$ is trivial.
\end{lem}

\begin{proof}
Set $h:=|{\mathrm{H}}^2(\overline  G, k^\times)|$.
It follows from \cite[Proposition 3.2]{Ya64}
(see also \cite[Lemma 5]{Ko82} and
\cite[Lemma 3.5.4(ii)]{NT89}) that
$|{\mathrm{H}}^2(\overline  G,k^\times)|=|M(\overline  G)|_{2'}$
since $k$ has characteristic $2$.

If \textbf{(K1)} holds, then $h=1$ by \cite[Theorem 2.7.4]{NT89}. 
Otherwise, by Lemma~\ref{G_C2xC2}, we have  $\overline{G}\cong
{\mathrm{PSL}}(2,q)\rtimes C_f$, where $q$ 
is a power of an odd prime with $q\equiv\pm 3$ (mod 8),     
and $f$ is odd. 
First, note that $|M({\mathrm{PSL}}(2,q))|=2$,  
see \cite[Theorem 4.9.1(ii)]{Ka85}.
If $q=3$, then $\overline G\cong{\mathrm{PSL}}(2,3)\cong\mathfrak A_4$
and $f=1$, so that the assertion holds.

So we may assume $q>3$, and hence ${\mathrm{PSL}}(2,q)$ is 
non-abelian simple as is well known.
Since $\overline G/{\mathrm{PSL}}(2,q)$ is cyclic and
${\mathrm{PSL}}(2,q)$ is perfect, 
it follows from \cite[Theorem 3.1(i)]{Jo74}
that $|M(\overline  G)|\,{\Big|}\,|M({\mathrm{PSL}}(2,q))| = 2$,
so that $h=1$.
\end{proof}

\begin{lem}\label{GInvariant_K} 
Suppose that $\theta\in{\mathrm{Irr}}(H)$ and that
$V$ is an indecomposable $kG$-module such that $V{\downarrow}_H$
contains $\theta$ as a constituent.
If $\overline G\ {\not\cong}\ \mathfrak A_5$ and $V$ is endo-trivial, 
then $\theta$ is $G$-invariant.
\end{lem}

\begin{proof}
If $\overline G\ {\not\cong}\ \mathfrak A_5$, 
then by Lemma~\ref{BenderSuzuki_K}, $G$ has no strongly $2$-embedded subgroups.
Therefore the claim follows from Lemma \ref{GInvariantGeneral}.
\end{proof}

\begin{prop}\label{H=1_K}
If $H=1$, then one of the following five cases holds:
\begin{enumerate}[\rm(a)]
\item If {\bf (K1)} holds, then  
$K(G)=\{[k_G]\}$.
\smallskip
\item If {\bf (K2)} holds, then 
$K(G)=X(G)=\{[k], [1_\omega], [1_{\omega^2}]\} \cong\mathbb Z/3\mathbb Z$, 
where $1_\omega$ and $1_{\omega^2}$ are the two non-trivial 
one-dimensional $kG$-modules.
\smallskip
\item If {\bf (K3)} holds, then 
$K(G)=\{[k], [5i],[(5i)^*]\}\cong K(\mathfrak A_4)
\cong\mathbb Z/3\mathbb Z$, 
and where $5i$, $(5i)^*$ are  uniserial,
trivial source, and  endo-trivial $kG$-modules in $B_0(G)$,  
both affording the unique irreducible character of degree $5$, $\chi_5\in{\mathrm{Irr}}(G)$. 
(See \cite[Lemma 4.1]{KL14}.)

\smallskip
\item If {\bf (K4)} holds, then $K(G)\cong X(G)\oplus TT_0(G)$, 
where $X(G)\cong \mathbb Z/f\mathbb Z$ and 
$TT_0(G)=\langle [(q-1)/2]\rangle\cong\mathbb Z/3\mathbb Z$
for $(q-1)/2$ a simple trivial source 
endo-trivial $kG$-module
of dimension $(q-1)/2$ affording an irreducible character
$\chi_{(q-1)/2}\in{\mathrm{Irr}}(B_0(G))$.

\smallskip

\item If {\bf (K5)} holds, then 
$K(G)\cong X(G)\oplus TT_0(G)$,
where $X(G)\cong\mathbb Z/f\mathbb Z$
and $TT_0(G)=\langle [V]\rangle\cong\mathbb Z/3\mathbb Z$
for a  trivial source 
endo-trivial $kG$-module $V$ such that $V$ is uniserial of 
length $3$  with composition factors $((q-1)/2)a$, $k_G$, 
$((q-1)/2)b$, where $((q-1)/2)a$ and $((q-1)/2)b$ are 
non-isomorphic simple $kG$-modules in $B_0(G)$ 
of dimension $(q-1)/2$, and $V$ affords the Steinberg character 
$\textup{St}_G\in\mathrm{Irr}(B_0(G))$ of degree $q$.
\end{enumerate}
\end{prop}

\begin{proof}
(a) is clear since $G\cong C_2\times C_2$ is a $2$-group. 
Next $K(\fA_4)\cong K(\fA_5)\cong \IZ/3\IZ$ by 
\cite[Theorem 4.2 and its proof]{CMN09}. Moreover 
$K(\fA_4)\cong X(\fA_4)$ by Lemma~\ref{lem:TT(G)}(a) 
because $C_2\times C_2$ is normal in $\fA_4$, and 
the structure of the modules in $K(\fA_5)$ is given 
by  \cite[Lemma 4.1]{KL14}. Hence (b) and (c) hold.
\par
If {\bf (K4)} or {\bf (K5)} holds, then  
$K(G)\cong\mathbb Z/f\mathbb Z\oplus\mathbb Z/3\mathbb Z$ 
with $X(G)\cong \mathbb Z/f\mathbb Z$ and 
$TT_0(G)\cong \IZ/3\IZ$ by  \cite[Theorem 1.4(d) 
and its proof]{KL14}. 
Finally the structure of a generator $[V]$ of  
$TT_0(G)$ with $V$ indecomposable is obtained as follows. 
By \cite[Proposition 4.3(a)]{KL14}, 
the lift of $V$ to characteristic zero affords  
an irreducible character $\chi_V\in\Irr(B_0(G))$, 
so that using Theorem~\ref{thm:torchar}, by investigation of the generic character 
table of $\PSL(2,q)$ (see e.g. \cite[Table 5.4]{BonnafeSL2}) 
we see that $\chi_V(1)=(q-1)/2$ if 
{\bf (K4)} holds, and $\chi_V(1)=q$ if {\bf (K5)} holds. 
Then the composition factors of $V$  follow from the 
$2$-decomposition matrix of $\PSL(2,q)$, see 
\cite[Table 9.1]{BonnafeSL2}.
\end{proof}

\begin{proof}[Proof of Theorem~\ref{MainTheoremK}]
First $TT_0(G)\cong \IZ/3\IZ$ by 
\cite[Theorem 1.4 and its proof]{KL14}, and any 
indecomposable torsion endo-trivial module lifts 
to an irreducible ordinary character by 
\cite[Theorem 1.1(a)]{KL14}. In addition, since 
$P\cong C_2\times C_2$, $TT(G)=K(G)$  by 
Lemma~\ref{lem:TT(G)}(c).\par
Assume that {\rm\bf{(K1)}} holds, 
then $G$ is $2$-nilpotent, 
and  \cite[Theorem]{NR12} yields $K(G)=X(G)$, 
as was conjectured in \cite[Conjecture 3.6]{CMT11}. Hence (a) holds.
\par
Assume that  {\bf (K2)} holds, then $G$ is solvable by the Feit-Thompson
Theorem, so that $K(G)=X(G)$ by \cite[Theorem 6.2(2)]{CMT11}. 
Hence (b) holds.
\par
Assume that {\bf (K3)} holds. Then, by 
Lemma~\ref{BenderSuzuki_K}, $G$ has a strongly $2$-embedded
subgroup $G_0$ such that  $G_0/O_{2'}(G_0)\cong\mathfrak A_4$. 
Therefore $K(G)\cong K(G_0)$ by Lemma~\ref{lem:spe} and 
$K(G_0)\cong X(G_0)$ by (b).
The non-trivial indecomposable torsion endo-trivial $kG$-modules 
in $B_0(G)$ have dimension~$5$ by Proposition~\ref{H=1_K}(c). 
Hence (c) holds.\par
Assume that {\bf (K4)} or {\bf (K5)} holds. Then  $G$ has 
no strongly $2$-embedded subgroups, and $H^2(G,k^{\times})=1$ 
by Lemmas~\ref{BenderSuzuki_K}~and~\ref{Multiplier_K}.  Therefore, 
by Theorem~\ref{cor:KXK}, 
$$K(G)= X(G)+K(\overline G)\,,$$
where we identify 
$K(\overline G)$ with $\Inf_{\overline{G}}^{G}(K(\overline G))$.
In addition, by Proposition~\ref{H=1_K}{(d) and (e)},  we have 
$K(\overline{G})=X(\overline{G})
\oplus TT_0(\overline{G})\cong\mathbb Z/f\mathbb Z
\oplus\mathbb Z/3\mathbb Z$. 
Therefore
$$K(G)=X(G)+X(\overline{G})+TT_0(\overline{G})=X(G)+TT_0(G)$$
since $\Inf_{\overline{G}}^{G}(TT_0(\overline{G}) )=TT_0(G)$ and 
$\Inf_{\overline{G}}^{G}(X(\overline{G}))\leq X(G)$.
But $X(G)\cap TT_0(G)$ $=\{[k_G]\}$ by 
Proposition~\ref{H=1_K}(d) and (e), thus 
$$
K(G)\cong X(G)\oplus TT_0(G)\cong X(G)\oplus\IZ/3\IZ\,.
$$ 
Hence (d) holds.
\end{proof}


\vspace{1cm}

\noindent
{\bf Acknowledgements.}
{\small 
The authors are grateful to Ron Solomon
for answering a question on finite groups, to 
Gunter Malle for useful comments on a preliminary version of this text. 
They also wish to thank Burkhard K\"ulshammer, 
Geoffrey Robinson, Andrei Marcus, and Rebecca Waldecker 
for useful discussions.
The first author is grateful to the Swiss National Foundation of 
Sciences for supporting his visit at the TU Kaiserslautern in 
September 2014 through the second author's Fellowship for 
Prospective Researchers PBELP2$_{-}$143516,
and also to Gunter Malle for supporting his visit at the
TU Kaiserslautern in May 2015.}



\end{document}